\documentclass[11pt,a4paper]{amsart}

\usepackage{amsmath,amsthm,amssymb,amsfonts,mathrsfs,bbm}

\usepackage{graphicx} 
\usepackage{booktabs} 
\usepackage{lmodern} 
\usepackage{enumitem} 
\usepackage{multicol} 
\usepackage{setspace} 
\usepackage{tikz-cd} 
\usepackage{url} 
\usepackage{color,xcolor} 
\usepackage{comment}

\allowdisplaybreaks 

\newtheorem{thm}{Theorem}[section]
\newtheorem{prop}[thm]{Proposition}
\newtheorem{cor}[thm]{Corollary}

\newtheorem{qst}[thm]{Question}

\theoremstyle{definition}
\newtheorem{defn}[thm]{Definition}
\newtheorem*{defn*}{Definition}
\newtheorem{rmrk}[thm]{Remark}
\newtheorem{rmrks}[thm]{Remarks}

\newtheorem{claim}{Claim}
\newtheorem*{claim*}{Claim}

\newcommand{\ord}{\textnormal{Ord}}

\newcommand{\forces}{\Vdash}

\newcommand{\powerset}{\mathcal{P}}
\DeclareMathOperator{\dom}{dom}
\DeclareMathOperator{\ran}{ran}

\DeclareMathOperator{\crit}{crit}

\DeclareMathOperator{\Add}{Add}

\newcommand{\bbF}{\mathbb{F}}

\newcommand{\bbL}{\mathbb{L}}

\newcommand{\p}{\mathbb{P}}
\newcommand{\q}{\mathbb{Q}}
\newcommand{\bbR}{\mathbb{R}}

\newcommand{\dotq}{\dot{\mathbb{Q}}}
\newcommand{\dotp}{\dot{\mathbb{P}}}

\newcommand{\dotr}{\dot{\mathbb{R}}}

\usepackage{hyperref}
\hypersetup{colorlinks=true,allcolors=blue}
\usepackage{hypcap}

\title{Strongly compact cardinals and the continuum function}
\author{Arthur W. Apter}
\address{Department of Mathematics, Baruch College, City University of New York, New York NY 10010, USA
  \&  Department of Mathematics,  CUNY Graduate Center, 365 Fifth Avenue, New York NY 10016, USA} 

\email{awapter@alum.mit.edu} 

\urladdr{http://faculty.baruch.cuny.edu/aapter}

\author{Stamatis Dimopoulos}

\address{School of Mathematics, University of Bristol, Bristol BS8 1TW, England}

\email{stamatiosdimopoulos@gmail.com}

\urladdr{https://st-dimopoulos.github.io/}

\author{Toshimichi Usuba}
\address{Department of Pure and Applied Mathematics, Faculty of Science and Engineering,
Waseda University, Okubo 3-4-1, Shinjuku, Tokyo, 169-8555 Japan}

\email{usuba@waseda.jp}

\date{January 17, 2019}

\begin{document}

\begin{abstract}
We study the general problem of the behaviour of the continuum function
in the presence of non-supercompact strongly compact cardinals.
We begin by showing that it is possible to force violations of GCH
at an arbitrary strongly compact cardinal 
using only strong compactness as our initial assumption.
This result is due to the third author.
We then investigate realising Easton functions at and above
the least measurable limit of supercompact cardinals
starting from an initial assumption of the existence
of a measurable limit of supercompact cardinals. By results due to Menas,
assuming $2^\kappa = \kappa^+$,
the least measurable limit of supercompact cardinals $\kappa$
is provably in ZFC a non-supercompact strongly compact cardinal
which is not $\kappa^+$-supercompact.
We also consider generalisations of our earlier theorems
in the presence of more than one strongly compact cardinal.
We conclude with some open questions.
\end{abstract}

\maketitle

\section{Introduction}\label{sec:intro}
Easton's theorem (see \cite[Theorem 15.18]{jech})
was a milestone in set theory, which showed that ZFC by itself does not impose severe limitations on the behaviour of the continuum function at regular cardinals. 
However, when we bring large cardinals into the picture, the situation is more complicated. Often the mere violation of GCH at a large cardinal requires strong assumptions. 
The prototypical example is the case of a measurable cardinal.
By results of Gitik \cite{gitik-negation-sch, gitik-equiconsistency} 
(see also Mitchell \cite{mitchell-core-i}), the violation of GCH at a measurable cardinal
is equiconsistent with the existence of a measurable cardinal
$\kappa$ such that $o(\kappa) = \kappa^{++}$.

In this paper, we look at the possible behaviour of the continuum function in the presence of strongly compact cardinals that are not supercompact.
Our goal will be to work with strongly compact cardinals which possess no
non-trivial degrees of supercompactness. 
There are fundamental open questions in this regard, 
such as whether it is possible to force GCH at an arbitrary non-supercompact
strongly compact cardinal.

As motivation, let us mention that if we allow 
enough supercompactness assumptions, 
the continuum function at a non-supercompact strongly compact cardinal
can be manipulated fairly easily. 
For instance, to realise a $\Delta_2$-definable Easton function $F$, we can use a result due to Menas \cite[Theorem, pages 83--88]{menas-spct}, 
which shows that it is possible to realise $F$ while preserving the supercompactness of a cardinal $\kappa$. We can then use Magidor's Prikry iteration \cite{magidor-identity-crises}
that destroys all measurable cardinals below $\kappa$. 
In the resulting model, $\kappa$ is strongly compact, $\kappa$ is
the least measurable cardinal (and so is not $2^\kappa$-supercompact),
and $F$ is still realised. 

In a similar vein, suppose that $F$ is an Easton
function definable by a $\Delta_2$ formula
in a model $V$ of ZFC + GCH in which $\lambda$ is a supercompact
limit of supercompact cardinals.
The aforementioned theorem of Menas shows that it is possible
to force over $V$ to obtain a model $V_1$ in which
$\lambda$ remains a supercompact limit of supercompact cardinals
and the Easton function $F$ has been realised.

In $V_1$, let $\kappa < \lambda$ be the least measurable
limit of supercompact cardinals.
Another theorem of Menas shows that in $V_1$,
$\kappa$ is both strongly compact and not $2^\kappa$-supercompact.
In particular, by starting with hypotheses stronger than the
existence of a measurable limit of supercompact cardinals,
it is possible to force and construct a model containing a 
non-supercompact strongly compact cardinal in which $F$ has been realised.

Also, if we assume that the strongly compact cardinal has a sufficient degree of supercompactness, 
there are positive results. In \cite[Theorem 4.5]{hamkins-lottery}, 
Hamkins shows that if $\kappa$ is both strongly compact and
$\lambda$-supercompact, then
$\kappa$ can be forced to have 
its strong compactness and $\lambda$-supercompactness
indestructible under any $\kappa$-directed closed forcing that has size at most $\lambda$. 
In particular, it is possible to realise suitable Easton functions in the interval $[\kappa,\lambda)$.

The structure of this paper is as follows.
Section \ref{sec:intro} contains our introductory remarks.
Section \ref{sec:prelim} contains a discussion of our notation, terminology,
and some earlier results 
used later on.
We then separate our results into two categories, 
depending on whether we are interested in preserving a single strongly compact cardinal or 
more than one strongly compact cardinal. 
Our results for one strongly compact cardinal 
are found in Section  \ref{sec:local}.
We first answer a long-standing open question 
on the problem of whether it is possible to violate GCH at a strongly compact
cardinal using no stronger assumptions.
We show that just assuming $2^\kappa =
\kappa^+$ and $\kappa$ is strongly compact, it is possible to preserve 
the strong compactness of $\kappa$ while forcing any desired value for $2^\kappa$.
This result is due to the third author. 
We then address the question of what sort of 
Easton functions can be realised in the presence of a certain
non-supercompact strongly compact cardinal. 
We show that if $\kappa$ is the least measurable limit of supercompact cardinals 
and $F$ is an arbitrary Easton function defined on regular cardinals greater 
than or equal to $\kappa$, 
then it is possible to force to realise $F$ while preserving the 
fact that $\kappa$ is the least measurable limit of supercompact cardinals. 
The techniques used, however, will of necessity destroy many supercompact cardinals.
We therefore also present another result along the same lines,
where the Easton function realised has restrictions placed on it,
but all supercompact cardinals are preserved.

Our results for more than one supercompact cardinal 
appear in Section \ref{sec:global}.
We begin by showing how 
to iterate the partial orderings used in Section \ref{sec:local}
so as to preserve all measurable limits of supercompact cardinals simultaneously, 
while realising certain Easton functions 
at all of them.
We then prove a theorem which gives a partial answer to the problem
of the simultaneous preservation of all supercompact and measurable limits
of supercompact cardinals while violating GCH at each of them.
Finally, Section \ref{sec:ques} contains some open questions.

In order to present our results in full generality, we will make the minimal number
of assumptions on the structure of the class of strongly compact and supercompact
cardinals in our ground model. However, if we force over a model in which GCH 
and the property of {\em compactness coincidence} both hold\footnote{The
property of {\em compactness coincidence} states that the strongly compact and supercompact
cardinals coincide, except at measurable limit points.
Models satisfying compactness coincidence non-trivially were first constructed
by Kimchi and Magidor in \cite{kimchi-magidor}.
As we observe in the paragraph immediately preceding the statement of
Theorem \ref{thm:least-measurable-limit}, by work of Menas, a further
coincidence between these two classes is impossible.}
(such as the one constructed by the first author and Shelah
in \cite{apter-shelah-strong-equality}), then all strongly compact cardinals will be
preserved to our generic extension.
This is since the only strongly compact cardinals which exist
in a model satisfying compactness coincidence 
are the supercompact
cardinals and the measurable limits of supercompact cardinals.

\section{Preliminaries}\label{sec:prelim}

Our notation and terminology on forcing are
standard and follow \cite{cummings-handbook}. 
In particular, $p\leq q$ means {\em $p$ is stronger than $q$}, 
and we call a partial ordering {\em $\kappa$-directed closed}
if every directed subset of size less than $\kappa$ has a lower bound. 

We will say that {\em F is an Easton function for the
model V of ZFC} if 
$F$ satisfies the following conditions:

\begin{itemize}

\item Either $F \in V$ (if $F$ is a set) or $F$ is
definable over $V$ (if $F$ is a proper class).

\item $\dom(F)$ is a class of $V$-regular cardinals.

\item ${\rm rge}(F)$ is a class of $V$-cardinals.

\item For every $\kappa \in \dom(F)$, $F(\kappa) > \kappa$.

\item If $\kappa < \lambda$, $\kappa, \lambda \in \dom(F)$, $F(\kappa) \le F(\lambda)$.

\item For every $\kappa \in \dom(F)$, ${\rm cf}(F(\kappa)) > \kappa$.

\end{itemize}

\noindent A model $V^*$ of ZFC {\em realises an Easton function $F$} if
in $V^*$, for every regular cardinal $\delta$ in the domain of $F$,
$2^\delta = F(\delta)$.

We assume that 
the reader is familiar with the large cardinal notions of measurability, strong compactness, 
and supercompactness. See \cite{jech} for further details. As it is
a lesser known notion, we recall that a cardinal $\kappa$ is called \emph{tall} 
if for every $\lambda\geq\kappa$, 
there is an elementary embedding $j:V\to M$ with $\crit(j)=\kappa$, 
$j(\kappa)>\lambda$, 
and ${}^\kappa M\subseteq M$. 
In \cite{hamkins-tall}, Hamkins made a systematic study of tall cardinals.
We will use the following facts about tallness.

\begin{prop}{\rm (}\cite[Corollary 2.7]{hamkins-tall}{\rm )}\label{prop:tall-limit}
If $\kappa$ is measurable and a limit of tall cardinals, then $\kappa$ is tall.
\end{prop}

\begin{prop}{\rm (}\cite[Corollary 2.6]{hamkins-tall}{\rm )}\label{prop:nice-emb}
If $\kappa$ is tall, then for every $\lambda\geq \kappa$,
there is a $\lambda$-tallness embedding $j:V\to M$ with $\crit(j)=\kappa$ such that there is no $\lambda$-tall cardinal in $[\kappa,\lambda]$ in $M$. 
\end{prop}

When it comes to strong compactness, we are interested in functions with the Menas property, which is defined as follows.

\begin{defn}\label{defn:menas-property}
Suppose $\kappa$ is a strongly compact cardinal. A function $f:\kappa\to \kappa$ has the \emph{Menas property} if for all $\lambda\geq\kappa$, there is a $\kappa$-complete, 
fine ultrafilter $U$ on $\powerset_\kappa \lambda$ such that for the ultrapower embedding $j_U : V \to M_U$,
$|[id]_U]|^{M_U}<j_U(f)(\kappa)$ holds in $M_U$.
\end{defn}

First used by Menas in \cite{menas-strong-compactness}, this property is quite helpful when lifting strong compactness embeddings through forcing. In \cite{hamkins-lottery}, Hamkins showed that fast function forcing at an arbitrary strongly compact cardinal adds a function with the Menas property.

\begin{prop}{\rm (}\cite[Theorem 1.7]{hamkins-lottery}{\rm )}\label{prop:fast-function-menas}
Suppose $\kappa$ is a strongly compact cardinal. Then the fast function forcing $\bbF_\kappa$ preserves the strong compactness of $\kappa$ and adds a fast function $f:\kappa\to \kappa$ that has the Menas property.
\end{prop}

Moreover, there are cases when ZFC implies the existence of such a function.

\begin{prop}{\rm (}\cite[Theorem 2.21 
and Proposition 2.31]{menas-strong-compactness}{\rm )}\label{prop:menas-property}
Suppose $\kappa$ is a measurable cardinal which is a limit of strongly compact cardinals. 
Then $\kappa$ is strongly compact, and
the function $f:\kappa\to \kappa$, where $f(\alpha)$ is the least strongly compact cardinal greater than $\alpha$, has the Menas property.
\end{prop}

By an easy adaptation of the previous proposition, 
we can also obtain the following corollary, which will be used in our results.

\begin{cor}\label{cor:menas-property}
Suppose $\kappa$ is a measurable cardinal which is a limit of supercompact cardinals. 
Then $\kappa$ is strongly compact, and
the function $f:\kappa\to \kappa$, where $f(\alpha)$ is the least supercompact cardinal greater than $\alpha$, has the Menas property.
\end{cor}


We will want to show at certain junctures that no
new instances of large cardinals are created
in certain forcing extensions. 
This will follow by a corollary of 
Hamkins' work of \cite{hamkins-approximation}
on the approximation and cover properties
(which is a generalization of his
gap forcing results found in \cite{hamkins-gap-forcing}).
This corollary follows from 
\cite[Theorem 3 and Corollary 14]{hamkins-approximation}.
We therefore state as a
separate theorem
what is relevant for this paper, along
with some associated terminology,
quoting from \cite{hamkins-gap-forcing, hamkins-approximation}
when appropriate.
Suppose $\p$ is a partial ordering
which can be written as
${\mathbb Q} \ast \dot {\mathbb R}$, where
$|{\mathbb Q}| \le \delta$,
${\mathbb Q}$ is non-trivial, and
$\forces_{\mathbb Q} ``\dot {\mathbb R}$ is
$\delta^+$-directed closed''.
In Hamkins' terminology of
\cite{hamkins-approximation}, 
$\p$ {\em admits a closure point at $\delta$}.
Also, as in the terminology of
\cite{hamkins-gap-forcing, hamkins-approximation} and elsewhere,
an embedding
$j : V \to \bar M$ is
{\em amenable to $V$} when
$j \restriction A \in V$ for any
$A \in V$.
The specific corollary of
Hamkins' work from
\cite{hamkins-approximation} 
we will be using
is then the following.

\begin{thm}\label{tgf}

{\bf(Hamkins)}
Suppose that $V[G]$ is a generic
extension obtained by forcing with $\p$
that admits a closure point 
at some regular $\delta < \kappa$.
Suppose further that
$j: V[G] \to M[j(G)]$ is an elementary embedding
with critical point $\kappa$ for which
$M[j(G)] \subseteq V[G]$ and
${}^\delta{M[j(G)]} \subseteq M[j(G)]$ in $V[G]$.
Then $M \subseteq V$; indeed,
$M = V \cap M[j(G)]$. If the full embedding
$j$ is amenable to $V[G]$, then the
restricted embedding
$j \restriction V : V \to M$ is amenable to $V$.
If $j$ is definable from parameters
(such as a measure or extender) in $V[G]$,
then the restricted embedding
$j \restriction V$ is definable from the names
of those parameters in $V$.

\end{thm}

It immediately follows from Theorem \ref{tgf}
that any cardinal $\kappa$ which is either 
$\lambda$-supercompact or measurable
in a forcing extension obtained
by a partial ordering that admits a closure point
below $\kappa$ (such as at $\omega$)
must also be 
$\lambda$-supercompact or measurable
in the ground model $V$.
In particular, if $\bar V$ is a forcing extension
of $V$ by a partial ordering admitting a closure point
at $\omega$
in which each supercompact cardinal and
each measurable limit of supercompact cardinals is preserved,
the classes of supercompact cardinals 
and measurable limits of supercompact cardinals in $\bar V$ remain
the same as in $V$.



\section{Results for one strongly compact cardinal}\label{sec:local}

We begin by showing that we can force violations of GCH 
at a strongly compact cardinal $\kappa$ without any stronger assumptions. 
Theorem \ref{thm:sc-continuum} and Corollary \ref{cor:consistency} 
are due to the third author.
Here, $\Add(\kappa, \delta)$ is the standard partial ordering for
adding $\delta$ Cohen subsets of $\kappa$.

\begin{thm}{\bf (Usuba)}\label{thm:sc-continuum}
Let $\kappa$ be a strongly compact cardinal.
There is then a forcing extension in which the strong compactness of $\kappa$ is indestructible under $\Add(\kappa,\delta)$ for all $\delta$.
\end{thm}

\begin{proof}
By forcing with the fast function forcing 
${\mathbb F}_\kappa$ if necessary, 
we can assume that there is a function $f^*:\kappa\to\kappa$ with the Menas property.

Define $\p=\langle \p_\alpha,\dotq_\beta\mid \beta<\alpha < \kappa\rangle$, 
an Easton support iteration of length $\kappa$, as follows.
Let $\p_0$ be the trivial forcing notion.
$\dotq_\alpha$ is then also a name for the trivial forcing notion, 
unless $\alpha$ is inaccessible and $f^*``\alpha\subseteq \alpha$.
In this case, $\dotq_\alpha$ is a name for the lottery sum
$$\bigoplus_{\beta<f^*(\alpha)} \Add(\alpha,\beta),$$
as defined in $V^{\p_\alpha}$.\footnote{If 
${\mathfrak A}$ is a collection of partial orderings, then
the {\em lottery sum} is the partial ordering
$\bigoplus {\mathfrak A} =
\{\langle \p, p \rangle \mid \p \in {\mathfrak A}$
and $p \in \p\} \cup \{1\}$, ordered with
$1$ above everything and
$\langle \p, p \rangle \le \langle \p', p' \rangle$ iff
$\p = \p'$ and $p \le p'$.
Intuitively, if $G$ is $V$-generic over
$\bigoplus {\mathfrak A}$, then $G$
first selects an element of
${\mathfrak A}$ (or as Hamkins says in \cite{hamkins-lottery},
``holds a lottery among the posets in
${\mathfrak A}$'') and
then forces with it. The
terminology ``lottery sum'' is due
to Hamkins, although the concept
of the lottery sum of partial
orderings has been around for quite
some time and has been referred to
at different junctures via the names
``disjoint sum of partial orderings'',
``side-by-side forcing'', and
``choosing which partial ordering to
force with generically''.}
Let $G\subseteq \p$ be $V$-generic. 
The arguments of \cite[Theorem 4.1]{hamkins-lottery}
show that $\kappa$ remains strongly compact in $V[G]$. We wish to show that in $V[G]$, 
the strong compactness of $\kappa$ is indestructible 
under $\Add(\kappa,\delta)$ for all $\delta$. 
Fix $\delta$, 
and let $g\subseteq \Add(\kappa,\delta)$ be $V[G]$-generic. If we let $Q=\bigcup g$, 
then $Q : \kappa\times \delta\to 2$ is a function.

Let $\lambda>\max(\kappa, \delta)$ be a regular cardinal, and 
fix a cardinal $\theta \ge 2^{\lambda^{< \kappa}}$.
Using the Menas property of $f^*$, let $j:V\to M$ be an ultrapower embedding by a $\kappa$-complete, 
fine ultrafilter $U\in V$ on $\powerset_\kappa\theta$ with $\crit(j)=\kappa$ such that $|[id]_U|^M<j(f^*)(\kappa)$. Since there is no source of confusion, we will drop the subscript from elements of $M$ and denote them as $[h]$. As usual, $j``\theta\subseteq [id]$, so $|\theta|^M \le |[id]|^M$.

\begin{claim}\label{claim:pi}
There is in $M$ a function $\pi:[id]\to \theta$ such that for all $\alpha<\theta$, $\pi(j(\alpha))=\alpha$.
\end{claim}

\begin{proof}
For each $\alpha<\theta$, let $g_\alpha:\powerset_\kappa\theta\to V$ be a function such that $[g_\alpha]=\alpha$. 
Without loss of generality, we can assume that $g_\alpha(x)$ is
defined for every $x \in \powerset_\kappa \theta$. 
Let $h:\powerset_\kappa\theta\to V$ be the function given by 
$$h(x)=\{\langle \alpha,g_\alpha(x)\rangle \mid \alpha\in x\}.$$
By its definition, $h(x)$ is a function with domain $x$.
It follows that $[h]$ 
is a function with domain $[id]$, 
and for each $\alpha<\theta$, $[h](j(\alpha))=[g_\alpha]=\alpha$. 
This completes the proof,
since we can easily use $[h]$ to define a function $\pi$ with the required properties.
\end{proof}

We now proceed by lifting $j$ through $\p\ast 
\dot \Add(\kappa,\delta)$. 
As usual, $j(\p)$ can be factorised as $\p\ast \dotq \ast \dotp_{tail}$, 
where $\dotq$ is a name for the lottery sum $\bigoplus_{\beta<j(f^*)(\kappa)}\Add(\kappa,\beta)$, 
and $\dotp_{tail}$ is a name for the remaining stages through $j(\kappa)$. 
Using $G$ as an $M$-generic filter for $\p$, we can form $M[G]$. 
Also, since $\delta< \theta \le |[id]|^M < j(f^*)(\kappa)$, 
we can choose to force below a condition in $\q=(\dotq)_G$ 
that opts for $\Add(\kappa,\delta)$. 
Thus, we can use $g$ as an $M[G]$-generic filter for $\q$. Furthermore, note that since $j(f^*)(\kappa)>|[id]|^M$, 
$\p_{tail} = (\dotp_{tail})_{G\ast g}$ is at least $(|[id]|^+)^M$-closed in $M[G][g]$.

Force over $V[G][g]$ to add a generic filter $G_{tail}$ for $\p_{tail}$. 
Using $G_{tail}$ as an $M[G][g]$-generic filter for $\p_{tail}$, since $j `` G \subseteq G$,
we can lift $j$ in $V[G][g][G_{tail}]$ to
$$j:V[G]\to M[j(G)],$$
where $j(G)=G\ast g \ast G_{tail}$. 
In order to further lift $j$ through $\Add(\kappa,\delta)$, 
we will use a master condition argument. 
Consider the function $\pi$ given by Claim \ref{claim:pi}, 
and note that $|[id]\cap j(\delta)|^M\leq |[id]^M|<j(\kappa)$. 
Define in $M[j(G)]$ a function $q:\kappa\times ([id]\cap j(\delta))\to 2$ 
given by $q(\langle \beta,\gamma \rangle)=
Q(\langle \beta,\pi(\gamma) \rangle)$ if 
$\pi(\gamma) < \delta$, 
and $0$ otherwise.
Clearly, $q$ is a condition in $j(\Add(\kappa,\delta))$.

\begin{claim}\label{claim:q}
$q\leq j(p)$ for all $p\in g$.
\end{claim}

\begin{proof}
By elementarity and the fact that $\crit(j)=\kappa$, for each $p\in g$, $j(p)$ is a function with domain $j``\dom(p)=\{\langle \beta,j(\gamma)\rangle \mid \langle \beta,\gamma\rangle\in \dom(p)\}$.
Hence, $\dom(j(p))\subseteq \dom(q)$. 
For $\langle \beta,j(\gamma)\rangle\in \dom(j(p))$, 
we have $j(p)(\langle \beta,j(\gamma) \rangle)=
p(\langle \beta,\gamma \rangle)=
Q(\langle \beta,\gamma \rangle)=
Q(\langle \beta,\pi(j(\gamma)) \rangle)=q(\langle \beta, j(\gamma) \rangle)$.
\end{proof}

Force over $V[G][g][G_{tail}]$ to add a generic filter 
$h\subseteq j(\Add(\kappa,\delta))$ containing $q$. 
By Claim \ref{claim:q}, we can lift $j$ in $V[G][g][G_{tail}][h]$ to
$$j:V[G][g]\to M[j(G)][h].$$ 

Let $\vec{X}=\langle X_\xi\mid \xi< 2^{\lambda^{<\kappa}}\rangle \in V[G][g]$ 
be an enumeration of $\powerset(\powerset_\kappa\lambda)^{V[G][g]}$. 
In $M[j(G)][h]$, consider the set $B=\{\xi\in [id]\mid [id]\cap j(\lambda)\in j(\vec{X})_\xi\}$. 
Since $\p_{tail}\ast j(\dot \Add(\kappa,\alpha))$ 
is at least $(|[id]|^+)^M$-closed in $M[G][g]$, 
$B\in M[G][g]\subseteq V[G][g]$. 
Hence, $W=\{X_\xi \in \powerset (\powerset_\kappa \lambda)^{V[G][g]}
\mid j(\xi)\in B\} \in V[G][g]$, 
and since $\theta \ge 2^{\lambda^{< \kappa}}$,
$W$ is easily seen to be a $\kappa$-complete, fine ultrafilter 
on $\powerset_\kappa\lambda$. 
Thus, $\kappa$ is $\lambda$-strongly compact in $V[G][g]$.
Since $\lambda$ can be chosen arbitrarily large, we have shown that $\kappa$ remains strongly compact in $V[G][g]$.
This completes the proof of Theorem \ref{thm:sc-continuum}.
\end{proof}

\begin{cor}\label{cor:consistency}
The existence of a strongly compact cardinal is equiconsistent with the existence of a strongly compact cardinal where GCH fails. 
In particular, assuming $2^\kappa = \kappa^+$ and $\kappa$ is strongly compact,
it is possible to force to preserve the strong compactness of $\kappa$ while
also forcing any desired value for $2^\kappa$.
\end{cor}

We now proceed by looking at a specific case of a non-supercompact strongly compact cardinal, the least measurable limit of supercompact cardinals.
By (the proof of) \cite[Theorem 2.22]{menas-strong-compactness}, if $\kappa$ is
the least measurable limit of supercompact cardinals, then $\kappa$
isn't $2^\kappa$-supercompact. Thus, if $2^\kappa = \kappa^+$,
then $\kappa$ isn't $\kappa^+$-supercompact, i.e.,
$\kappa$ exhibits no non-trivial degree of supercompactness.

\begin{thm}\label{thm:least-measurable-limit}
Suppose GCH holds, $\kappa$ is the least measurable limit of supercompact cardinals, 
and $F$ is an arbitrary Easton function. 
There is then
a forcing extension in which 
$\kappa$ remains the least measurable limit of supercompact cardinals,
$\kappa$ exhibits no non-trivial degree of supercompactness, 
and $F$ is realised at all regular cardinals greater than or equal to $\kappa$.
\end{thm}

\begin{proof}
Let $f:\kappa\to \kappa$ be the function where $f(\alpha)$ is the least supercompact cardinal greater than $\alpha$. 
We define
 $\p=\langle \p_\alpha,\dotq_\beta\mid \beta<\alpha < \kappa\rangle$, an Easton support iteration
 of length $\kappa$.
 We start by letting $\p_0 = \Add(\omega, 1)$. 
 For $0 \le \alpha < \kappa$, $\dotq_\alpha$ is then defined as follows:
\begin{enumerate}
  \item If $\alpha$ is supercompact, but not the least supercompact 
 cardinal greater than an inaccessible limit of supercompact cardinals, $\dotq_\alpha$ is a name for the Laver preparation 
 \cite{laver-preparation} 
 of $\alpha$, defined using only $\sigma$-directed closed partial orderings.
 Here, $\sigma<\alpha$ is the least inaccessible cardinal greater than the supremum of 
 the supercompact cardinals below $\alpha$, or the least inaccessible cardinal
 if there are no supercompact cardinals below $\alpha$.
 We explicitly note that since there is no
 supercompact limit of supercompact cardinals below $\kappa$,
 the first non-trivial stage in the realisation of $\dotq_\alpha$
 can be assumed not to occur until after stage $\sigma$.
  \item If $\alpha$ is an inaccessible limit of supercompact cardinals, 
  $\dotq_\alpha$ is a name for $\bigoplus_{\beta<f(\alpha)}\Add(\alpha,\beta)$.
  \item In all other cases, $\dotq_\alpha$ is a name for the trivial forcing notion.
\end{enumerate}

Let $G\subseteq \p$ be $V$-generic.
In $V[G]$, we force with the Easton product $$\prod_{\delta \ge \kappa} \Add(\delta, F(\delta)),$$ where
$\delta \ge \kappa$ is a ($V$ or $V[G]$)-regular cardinal.
Let $g \times H$ be $V[G]$-generic over
$$\Add(\kappa, F(\kappa)) \times \prod_{\delta > \kappa} \Add(\delta, F(\delta)) =
\Add(\kappa, F(\kappa)) \times {\mathbb R}  = \prod_{\delta \ge \kappa} \Add(\delta, F(\delta)).$$
Standard arguments (see \cite[proof of Theorem 15.18]{jech}) show that $F$ is realised in $V[G][g][H]$
at all cardinals greater than or equal to $\kappa$.
We wish to show that $\kappa$ remains the least measurable limit of
supercompact cardinals in $V[G][g][H]$ (so $\kappa$ is strongly compact in $V[G][g][H]$)
and also exhibits no non-trivial degree of supercompactness in $V[G][g][H]$.

We begin by showing that $\kappa$ remains a limit of supercompact cardinals in $V[G][g][H]$.
For this, note that in $V$, the set 
$\{\alpha<\kappa\mid \alpha$ is supercompact and 
is not the least supercompact cardinal
greater than an inaccessible limit of supercompact cardinals$\}$ 
is unbounded below $\kappa$. For each such $\alpha$, the 
partial ordering $\q_\alpha$ forces $\alpha$ to have its supercompactness
indestructible under $\alpha$-directed closed forcing.
Since all the stages of $\p$ above $\alpha$ are at least $\alpha$-directed closed, 
$\alpha$ remains indestructibly supercompact in $V[G]$. 
Moreover, since the forcing $\Add(\kappa,F(\kappa))\times {\mathbb R}$ 
is itself $\alpha$-directed closed in $V[G]$, $\alpha$ is supercompact in $V[G][g][H]$ as well.

We now show that $\kappa$ remains measurable in $V[G][g][H]$.
We first note that by Corollary \ref{cor:menas-property}, $f$ has the Menas property.
Further, in the definition of ${\mathbb P}$, if $\alpha$ is an inaccessible limit of supercompact cardinals,
the first non-trivial stage of forcing after stage $\alpha$ does not occur until after $f(\alpha)$.
Therefore, the proof of Theorem \ref{thm:sc-continuum} immediately yields that $\kappa$
is strongly compact in $V[G][g]$. However, by Easton's lemma \cite[Lemma 15.19]{jech},
$\bbR$ is $(\kappa^+, \infty)$-distributive in $V[G][g]$. This consequently implies that
$\kappa$ is measurable in $V[G][g][H]$.

To complete the proof of Theorem \ref{thm:least-measurable-limit}, it only remains to show that in
$V[G][g][H]$, $\kappa$ remains the least measurable limit of supercompact cardinals
and exhibits no non-trivial degree of supercompactness.
Towards a contradiction, suppose $\gamma < \kappa$ is a measurable limit of supercompact
cardinals in $V[G][g][H]$.
By its definition, ${\mathbb P} \ast (\dot {\prod_{\delta \ge \kappa} \Add(\delta, F(\delta)))}$ 
admits a closure point at $\omega$.
Hence, by our remarks in the paragraph immediately following Theorem \ref{tgf},
$\kappa$ exhibits no non-trivial degree of supercompactness in $V[G][g][H]$.  
In addition, these same remarks imply that
$\gamma$ must be a measurable limit of supercompact cardinals in $V$, which contradicts
that in $V$, $\kappa$ is the least measurable limit of supercompact cardinals.
This completes the proof of Theorem \ref{thm:least-measurable-limit}.
\end{proof}

\begin{rmrk}
Our choice of $\kappa$ as
the least measurable limit of supercompact cardinals together with GCH 
was in order to highlight the fact that we do not assume 
any non-trivial degree of supercompactness for $\kappa$. 
However, the same proof would go through if $\kappa$ were an arbitrary measurable limit of supercompact cardinals.
\end{rmrk}

\begin{rmrk}
The partial ordering ${\mathbb P}$ of Theorem \ref{thm:least-measurable-limit} will destroy the
supercompactness of any cardinal $\alpha < \kappa$ which is in $V$ the least supercompact cardinal 
above an inaccessible limit of supercompact cardinals.
To see this, note that we can write ${\mathbb P} = {\mathbb P}^1 \ast \dot{\mathbb P}^2$, where
$\p^1$ is forcing equivalent to a partial ordering having size less than $\alpha$,
and $\dot{\mathbb P}^2$ is forced to add a subset of some $\gamma > \alpha$,
$\gamma$ below the least supercompact cardinal in $V$ above $\alpha$.
By \cite[Theorem, page 550]{hamkins-shelah}, it is then the case that in
$V[G]$, $\alpha$ is no longer $\gamma$-supercompact.
Hence, by the closure properties of $\prod_{\delta \ge \kappa} \Add(\delta, F(\delta))$,
$\alpha$ is no longer $\gamma$-supercompact in $V[G][g][H]$ as well.
However, if we are willing to impose some restrictions on our Easton function,
it is possible to prove a version of Theorem \ref{thm:least-measurable-limit} in which
all supercompact cardinals below $\kappa$ are preserved. In particular, we have the following theorem.
\end{rmrk}

\begin{thm}\label{thm:least-measurable-limit-ii}
Suppose GCH holds, $\kappa$ is the least measurable limit of supercompact cardinals, 
and $F$ is an Easton function such that $F(\kappa)$ is regular and
$F(\delta) = F(\kappa)$ for every $\delta \in [\kappa, F(\kappa))$. 
There is then
a forcing extension in which 
$\kappa$ remains the least measurable limit of supercompact cardinals,
$\kappa$ exhibits no non-trivial degree of supercompactness, 
$F$ is realised at all regular cardinals greater than or equal to $\kappa$,
and the supercompact cardinals below $\kappa$ are the same as in the ground model. 
\end{thm}

\begin{proof}
We first note that since $\kappa$ is strongly compact, by \cite[Theorem 2.11]{hamkins-tall},
$\kappa$ is also tall.
Next, let $f:\kappa\to \kappa$ be the function where $f(\alpha)$ is the least tall cardinal greater than $\alpha$. 
Since every supercompact cardinal is clearly also tall and $\kappa$ is a limit of supercompact cardinals,
$f(\alpha)$ has a value less than $\kappa$ for every $\alpha < \kappa$
and so is well-defined.
It is also the case that $f$ has the {\em Menas property for tallness} \cite[page 75]{hamkins-tall}, i.e.,
for every ordinal $\theta$, there is an elementary embedding $j : V \to M$ with ${\rm crit}(j) = \kappa$,
${}^\kappa M \subseteq M$, and $j(f)(\kappa) \ge \theta$.
To see this, let $\theta \ge \kappa$. Using Proposition \ref{prop:nice-emb}, we
fix for $\kappa$ a $\theta$-tallness embedding $j:V\to M$ such that: 
\begin{itemize}
 \item $\crit(j)=\kappa$.
  \item $j(\kappa)>\theta$.
 \item ${}^\kappa M \subseteq M$.
 \item $j$ is given by a $(\kappa,\theta)$-extender embedding.
 \item There is no $\theta$-tall cardinal in $M$ in the interval $[\kappa,\theta]$.
\end{itemize}
Since in $M$, $j(f)(\kappa)$ is the least tall cardinal greater than $\kappa$,
and because there are no $\theta$-tall cardinals in $M$ in the interval $[\kappa,\theta]$,
it follows that $j(f)(\kappa) > \theta$.

We will now proceed 
along the same lines as the proof of Theorem \ref{thm:least-measurable-limit}.
We begin by
fixing for $\kappa$ an $F(\kappa)$-tallness embedding $j:V\to M$ such that: 
\begin{itemize}
 \item $\crit(j)=\kappa$.
 \item $j(\kappa)>F(\kappa)$.
 \item ${}^\kappa M \subseteq M$.
 \item $j$ is given by a $(\kappa,F(\kappa))$-extender embedding.
 \item There is no $F(\kappa)$-tall cardinal in the interval $[\kappa,F(\kappa)]$.
\end{itemize} 
We next define
 $\p=\langle \p_\alpha,\dotq_\beta\mid \beta<\alpha < \kappa\rangle$, an Easton support iteration
 of length $\kappa$.
 We start by letting $\p_0 = \Add(\omega, 1)$. 
 For $0 \le \alpha < \kappa$, $\dotq_\alpha$ is then defined as follows:
\begin{enumerate}
  \item If $\alpha$ is supercompact, 
 $\dotq_\alpha$ is a name for the Laver preparation 
 \cite{laver-preparation} 
 of $\alpha$, defined using only $\sigma$-directed closed partial orderings.
 Here, $\sigma<\alpha$ is redefined as 
 the least tall cardinal greater than $\gamma$, the supremum of 
 the supercompact cardinals below $\alpha$, or the least
 tall cardinal if there are no supercompact cardinals below $\alpha$.
 As before, since there is no supercompact limit of supercompact cardinals below $\kappa$
 and $\alpha < \kappa$ is supercompact, $\gamma < \alpha$.
 Also, by \cite[Lemma 2.1]{apter-cummings}, $\sigma \in (\gamma, \alpha)$
 (where we take $\gamma = \omega$ if there are no supercompact cardinals below $\alpha$).
 Therefore, in analogy to the proof of Theorem \ref{thm:least-measurable-limit},
 the first non-trivial stage in the realisation of $\dotq_\alpha$
 can be assumed not to occur until after stage $\sigma$.
  \item If $\alpha$ is an inaccessible limit of supercompact cardinals, 
  $\dotq_\alpha$ is a name for $\bigoplus_{\beta<f(\alpha)}\Add(\alpha,\beta)$.
  \item In all other cases, $\dotq_\alpha$ is a name for the trivial forcing notion.
\end{enumerate}

Let $G \subseteq \p$ be $V$-generic, and let $g_0 \subseteq \Add(\kappa, F(\kappa))$
be $V[G]$-generic.
By clause (2) in the definition of $\p$, the choice of $j$, 
and the fact $f$ has the Menas property for tallness,
it is possible to opt for $\Add(\kappa, F(\kappa))$ at stage $\kappa$ in $M$
in the definition of $j(\p)$.
In addition, by clause (1) in the definition of $\p$, the first non-trivial stage
in $M$ in the definition of $j(\p)$ after $\kappa$ does not occur until after stage $j(f)(\kappa)$,
the least tall cardinal in $M$ greater than $\kappa$.
This means that the proof of \cite[Theorem 3.13]{hamkins-tall} unchanged remains valid and
allows us to infer the existence of a $\kappa$-directed closed, $(\kappa^+, \infty)$-distributive,
cardinal and cofinality preserving partial ordering ${\mathbb R} \in V[G][g_0]$ such that
if $g_1 \subseteq {\mathbb R}$ is $V[G][g_0]$-generic, in $V[G][g_0][g_1]$,
$\kappa$ is a tall cardinal, and $2^\delta = F(\kappa)$ for every $\delta \in [\kappa, F(\kappa))$.
If we now let 
$g_2 \subseteq \prod_{\delta \ge F(\kappa)} \Add(\delta, F(\delta))$ (the Easton product for
$\delta \ge \kappa$ a regular cardinal in any of the models $V$, $V[G]$, $V[G][g_0]$, or $V[G][g_0][g_1]$) be
$V[G][g_0][g_1]$-generic, then $F$ is realised in $V[G][g_0][g_1][g_2]$ at all regular cardinals $\delta \ge \kappa$.
In addition, the same arguments as found in the proof of
  Theorem \ref{thm:least-measurable-limit} show that in $V[G][g_0][g_1][g_2]$, 
$\kappa$ remains the least measurable limit of supercompact cardinals, and
$\kappa$ exhibits no non-trivial degree of supercompactness.
By the definition of $\p$, all ground model supercompact cardinals less than $\kappa$
have been made indestructible and hence are preserved to $V[G][g_0][g_1][g_2]$.
Consequently, by our remarks in the paragraph immediately following Theorem \ref{tgf},
the supercompact cardinals below $\kappa$ in $V[G][g_0][g_1][g_2]$ are the
same as in $V$.
This completes the proof of Theorem \ref{thm:least-measurable-limit-ii}.
\end{proof}

\section{Results for more than one strongly compact cardinal 
}\label{sec:global}

In the previous section, we successfully violated GCH and even realised certain 
Easton functions above one non-supercompact strongly compact cardinal,
the least measurable limit of supercompact cardinals.
%
%
We now present results in which we handle more than one measurable
limit of supercompact cardinals.
In what follows, let $A = \{\delta \mid \delta$ is a measurable limit of supercompact cardinals$\}$.
Define $\Omega = \sup(A)$ if $A$ is a set, or $\Omega = \ord$ if $A$ is a proper class.
Let $f : \Omega \to \Omega$ be the function where $f(\alpha)$ is the least
supercompact cardinal greater than $\alpha$.

\begin{thm}\label{thm:global}
Suppose $V$ is a model of GCH containing more than one measurable
limit of supercompact cardinals.
Let $F$ be an Easton function defined on measurable limits of supercompact cardinals such that $F(\kappa)<f(\kappa)$ 
for any $\kappa\in \dom(F)$. 
Then there is a forcing extension in which the measurable limits of supercompact cardinals 
are the same as in $V$ and $F$ is realised. 
\end{thm}

%
%
\begin{proof}
Intuitively, we will proceed by iterating the forcing notion used
in the proof of Theorem \ref{thm:least-measurable-limit}. More formally, let
$\langle \kappa_\alpha \mid \alpha < \Omega \rangle$ enumerate in increasing
order the measurable limits of supercompact cardinals.
We define
 $\p=\langle \p_\alpha,\dotq_\beta\mid \beta<\alpha < \Omega\rangle$, an Easton support iteration
 of length $\Omega$.
 We start by letting $\p_0 = \Add(\omega, 1)$. 
 For $0 \le \alpha < \Omega$, $\dotq_\alpha$ is then defined as follows:
\begin{enumerate}
  \item If $\alpha$ is supercompact, but neither the least supercompact 
 cardinal greater than an inaccessible limit of supercompact cardinals
 nor a supercompact limit of supercompact cardinals, $\dotq_\alpha$ is a name for the Laver preparation 
 \cite{laver-preparation} 
 of $\alpha$, defined using only $\sigma$-directed closed partial orderings.
 Here, $\sigma<\alpha$ is the least inaccessible cardinal greater than the supremum of 
 the supercompact cardinals below $\alpha$, or the least inaccessible cardinal
 if there are no supercompact cardinals below $\alpha$.
 We explicitly note that since $\alpha$ is not a supercompact
 limit of supercompact cardinals,
 the first non-trivial stage in the realisation of $\dotq_\alpha$
 can be assumed not to occur until after stage $\sigma$.
  \item If $\alpha$ is a non-measurable inaccessible limit of supercompact cardinals, 
  $\dotq_\alpha$ is a name for $\bigoplus_{\beta<f(\alpha)}\Add(\alpha,\beta)$.
  \item If $\alpha$ is a measurable limit of supercompact cardinals, 
  $\dotq_\alpha$ is a name for $\Add(\alpha, F(\alpha))$.
  \item In all other cases, $\dotq_\alpha$ is a name for the trivial forcing notion.
\end{enumerate}

Let $G\subseteq \p$ be $V$-generic. 
Fix $\kappa$ such that $\kappa$ is a measurable limit
of supercompact cardinals in $V$. Write $G$ as
$G_\kappa \ast g \ast G_{tail}$, where $G_\kappa$ is $V$-generic for
the forcing defined through stage $\kappa$, $g$ is $V[G_\kappa]$-generic for
$\Add(\kappa, F(\kappa))$ (the stage $\kappa$ forcing), and
$G_{tail}$ is $V[G_\kappa][g]$-generic for the rest of $\p$.
The proof of Theorem \ref{thm:least-measurable-limit} shows that
in $V[G_\kappa][g]$, $\kappa$ remains a measurable limit of supercompact cardinals, and $2^\kappa = F(\kappa)$.
By the definition of $\p$, because $F(\kappa) < f(\kappa)$ and the first non-trivial stage of
forcing after stage $\kappa$ does not occur until after $f(\kappa)$,
in $V[G_\kappa][g][G_{tail}] = V[G]$, $\kappa$ remains a measurable limit of supercompact cardinals, 
and $2^\kappa = F(\kappa)$.
By our remarks in the paragraph immediately following Theorem \ref{tgf},
any cardinal in $V[G]$ which is a measurable limit of supercompact
cardinals must have been a measurable limit of supercompact cardinals in $V$.
Since standard arguments show that if $\p$ is a proper class, $V[G]$ is a model of ZFC,
this completes the proof of Theorem \ref{thm:global}.
\end{proof}

As was the case with Theorem \ref{thm:least-measurable-limit}, the proof of Theorem \ref{thm:global}
yields that any $\delta < \Omega$ which in $V$ is the least supercompact cardinal greater than an inaccessible
limit of supercompact cardinals has its supercompactness destroyed after forcing with $\p$.
It is possible, however, by making some slight changes in the definition of $\p$,
to prove an analogue of Theorem \ref{thm:global} in which
both the measurable limits of supercompact cardinals and the
supercompact cardinals which are not limits of supercompact cardinals are the same as in $V$.
In particular, suppose we assume: 
\begin{itemize}

\item $A$, $\Omega$, and $\langle \kappa_\alpha \mid \alpha < \Omega \rangle$
have been defined as in the proof of Theorem \ref{thm:global}.

\item $f : \Omega \to \Omega$ is redefined as $f(\alpha)$ is
the least tall cardinal greater than $\alpha$.

\item 
$\sigma$ is redefined as
the least tall cardinal greater than the supremum of the supercompact cardinals below $\alpha$,
or the least tall cardinal if there are no supercompact cardinals below $\alpha$.

\item We define a partial ordering $\p$ as in the proof of Theorem \ref{thm:global}, except that
in Case (1) of the definition of $\p$, 
$\alpha$ can be {\em any} supercompact cardinal which is not a limit of supercompact cardinals.

\end{itemize}
We now have the following.

\begin{thm}\label{thm:global-ii}
Suppose $V$ is a model of GCH containing more than one measurable
limit of supercompact cardinals.
Let $F$ be an Easton function defined on measurable limits of supercompact cardinals such that $F(\kappa)<f(\kappa)$ 
for any $\kappa\in \dom(F)$. 
Then there is a forcing extension in which 
the measurable limits of supercompact cardinals 
and the supercompact cardinals which are not limits of supercompact cardinals
are the same as in $V$,
and $F$ is realised. 
\end{thm}

The proof of Theorem \ref{thm:global-ii} is essentially the same as the proof of Theorem \ref{thm:global},
with all references to the proof of Theorem \ref{thm:least-measurable-limit}
replaced by references to the proof of Theorem \ref{thm:least-measurable-limit-ii}.
We note only that if $\delta < \Omega$ is in $V$ a supercompact cardinal which
is not a limit of supercompact cardinals, the definition of $\p$ 
(specifically, the change made in Case (1)) 
shows that $\delta$ is preserved to the generic extension $V[G]$ by $\p$.
By our remarks in the paragraph immediately following Theorem \ref{tgf},
it now immediately follows that the supercompact cardinals which
are not limits of supercompact cardinals are the same in both $V$ and $V[G]$.
If 
$A$ is a set instead
of a proper class (so that in particular, $\Omega$ is an ordinal), then
by the L\'evy-Solovay results \cite{levysolovay},
the supercompact cardinals above $\Omega$ in both $V$ and $V[G]$
are precisely the same.


The techniques used in the proofs of Theorems \ref{thm:global} and \ref{thm:global-ii} do not
seem to allow for the preservation of supercompact limits of supercompact cardinals.
Although we do not yet know a way of 
accomplishing this in general, it is possible
to achieve this goal in a certain restricted situation.
More specifically, we have the following.


\begin{thm}\label{thm:spct-limit}
Suppose $V$ is a model of GCH in which $\kappa$ is the only supercompact limit of supercompact cardinals 
and there is no inaccessible cardinal greater than $\kappa$. 
Then there is a forcing extension in which 
the supercompact cardinals and measurable limits of supercompact cardinals are the same as in $V$ 
(so in particular, $\kappa$ remains the only supercompact limit of supercompact cardinals), 
and $2^\delta=\delta^{++}$ for every 
$\delta$ which is either supercompact or a measurable limit of supercompact cardinals.
\end{thm}

\begin{proof}
Let $\langle \kappa_\alpha \mid \alpha \le \kappa \rangle$ enumerate in increasing
order the measurable limits of supercompact cardinals.
Let $f : \kappa \to \kappa$ be the function where $f(\alpha)$ is the least
tall cardinal greater than $\alpha$.
We define
 $\p=\langle \p_\alpha,\dotq_\beta\mid \beta<\alpha \le \kappa\rangle$, an Easton support iteration
 of length $\kappa + 1$.
 We start by letting $\p_0 = \Add(\omega, 1)$. 
 For $0 \le \alpha \le \kappa$, $\dotq_\alpha$ is then defined as follows:
\begin{enumerate}
  \item If $\alpha < \kappa$ is supercompact, 
  $\dotq_\alpha$ is a name for the Laver preparation 
 \cite{laver-preparation} 
 of $\alpha$, defined using only $\sigma$-directed closed partial orderings.
 Here, $\sigma<\alpha$ is the least tall cardinal greater than the supremum of 
 the supercompact cardinals below $\alpha$,
 or the least tall cardinal if there are no supercompact cardinals below $\alpha$.
 We explicitly note that as in the proof of
 Theorem \ref{thm:least-measurable-limit-ii}, since $\alpha$ is not a supercompact
 limit of supercompact cardinals,
 the first non-trivial stage in the realisation of $\dotq_\alpha$
 can be assumed not to occur until after stage $\sigma$.
  \item If $\alpha$ is a non-measurable inaccessible limit of supercompact cardinals, 
  $\dotq_\alpha$ is a name for $\bigoplus_{\beta<f(\alpha)}\Add(\alpha,\beta)$.
  \item If $\alpha \le \kappa$ is either supercompact or a measurable limit of supercompact cardinals, 
  $\dotq_\alpha$ is a name for $\Add(\alpha, \alpha^{++})$.
  \item In all other cases, $\dotq_\alpha$ is a name for the trivial forcing notion.
\end{enumerate}

Let $G \subseteq \p$ be $V$-generic. Write
$G$ as $G_\kappa \ast g$, where $G_\kappa$ is $V$-generic
for the forcing defined through stage $\kappa$, and $g$ is
$V[G_\kappa]$-generic for $\Add(\kappa, \kappa^{++})$ (the stage $\kappa$ forcing).
The arguments found in the proofs of 
Theorems \ref{thm:global} and \ref{thm:global-ii}, 
in tandem with the definition of $\p$, show that in $V[G_\kappa][g]$,
the supercompact cardinals less than $\kappa$
and measurable limits of supercompact cardinals are the same as in $V$, and
$2^\delta = \delta^{++}$ for every $\delta$ which is either supercompact or a measurable
limit of supercompact cardinals. In addition, since there are no inaccessible cardinals
greater than $\kappa$ in $V$, there are no inaccessible cardinals greater than $\kappa$
in $V[G_\kappa][g]$ as well.
Thus, the proof of Theorem \ref{thm:spct-limit} will be complete once
we have shown that $\kappa$ remains supercompact in $V[G_\kappa][g]$.

To do this,
fix an arbitrary
$\lambda \ge \kappa^{++}$, 
and let $j:V\to M$ be a 
$\gamma = 2^{\lambda^{< \kappa}}$-supercompactness embedding 
with $\crit(j)=\kappa$. 
In $V$, there is no inaccessible cardinal greater than $\kappa$, 
and since ${}^\gamma M\subseteq M$, 
in $M$, there is no inaccessible cardinal in 
$(\kappa,\gamma]$. 
Thus, we can write $j(\p)$ as $\p\ast \dot \Add(\kappa,\kappa^{++})\ast \dot\p_{tail}
\ast \dot \Add(j(\kappa), j(\kappa^{++}))$,
where $\dot\p_{tail}$ is a name for a $\gamma^+$-directed closed forcing
whose first non-trivial stage occurs after $\gamma^+$.
Standard arguments (see, e.g., \cite[proof of the Theorem, pages 387--388]{laver-preparation})
now show that $\kappa$ is $\lambda$-supercompact in $V[G_\kappa][g]$.
Since $\lambda$ was chosen arbitrarily, $\kappa$ is supercompact in $V[G_\kappa][g]$.
This completes the proof of Theorem \ref{thm:spct-limit}.
\end{proof}

\section{Questions}\label{sec:ques}


The following questions remain open concerning strongly compact
cardinals and the continuum function.

\begin{qst}
If $\kappa$ is a strongly compact cardinal, can we force GCH at $\kappa$ while preserving the strong compactness of $\kappa$ without assuming any stronger hypotheses?
\end{qst}

\begin{qst}[Woodin]
If GCH holds below a strongly compact cardinal, does it hold above it too?
\end{qst}

Also, 
our methods leave unresolved the problem of realising
an arbitrary Easton function in the presence of a strongly compact cardinal.

\begin{qst}
Suppose $F$ is {\em any}
Easton function and $\kappa$ is a strongly compact cardinal. 
Under what conditions can we realise $F$ while preserving the strong compactness of $\kappa$?
\end{qst}

One of the challenges in the proofs of the theorems of Section \ref{sec:global} 
that remains unresolved is the 
preservation of {\em arbitrary} supercompact limits of supercompact cardinals.

\begin{qst}
Can we prove analogues of Theorems \ref{thm:global}, \ref{thm:global-ii}, and \ref{thm:spct-limit} 
where {\em all} supercompact limits of supercompact cardinals are preserved?
\end{qst}

\bibliographystyle{abbrvurl}
\bibliography{math_refs}

\end{document}